\documentclass[12pt,a4paper,reqno]{amsart}
\usepackage[all]{xy}
\usepackage{amsmath}
\usepackage{amsfonts}
\usepackage{amssymb,amsthm,amsfonts,amsthm,latexsym,enumerate,url,cases}
\numberwithin{equation}{section}
\usepackage{mathrsfs}
\usepackage{hyperref}
\usepackage{extarrows}
\hypersetup{colorlinks=true,citecolor=blue,linkcolor=blue,urlcolor=blue}
\def\pmod #1{\ ({\rm{mod}}\ #1)}

\newtheorem{theorem*}{Theorem}
\newtheorem{lemma*}{Lemma}
\theoremstyle{plain}

\newtheorem{theorem}{Theorem}

\theoremstyle{definition}


\begin{document}

\title
[{On a conjecture of R. M. Murty and V. K. Murty}] {On a conjecture of R. M. Murty and V. K. Murty}

\author
[Yuchen Ding] {Yuchen Ding}

\address{(Yuchen Ding) School of Mathematical Science,  Yangzhou University, Yangzhou 225002, People's Republic of China}
\email{ycding@yzu.edu.cn}


\begin{abstract}
Let $\omega^*(n)$ be the number of primes $p$ such that $p-1$ divides $n$. Recently, R. M. Murty and V. K. Murty proved that
$$x(\log\log x)^3\ll\sum_{n\le x}\omega^*(n)^2\ll x\log x.$$
They further conjectured that there is some positive constant $C$ such that
$$\sum_{n\le x}\omega^*(n)^2\sim Cx\log x$$
as $x\rightarrow \infty$. In this short note, we give the correct order of the sum by showing that $$\sum_{n\le x}\omega^*(n)^2\asymp  x\log x.$$
\end{abstract}
\maketitle

\baselineskip 18pt

Let $\omega(n)$ be the number of distinct prime divisors of $n$. In about one hundred years ago, Hardy and Ramanujan \cite{HR} found out that $\omega(n)$ has normal order $\log\log n$ which means that for all most all integers $n$ we have $\omega(n)\sim \log\log n$. Later, Tur\'an \cite{Tu} provided a quite elegantly simplified proof by establishing $$\sum_{n\le x}(\omega(n)-\log\log n)^2\ll x\log\log x.$$
In 1955, Prachar \cite{Pr} considered a variant arithmetic function of $\omega$. Let $\omega^*(n)$ be the number of primes $p$ such that $p-1$ divides $n$. Prachar proved that
$$\sum_{n\le x}\omega^*(n)=x\log\log x+Bx+O(x/\log x)$$
and
$$\sum_{n\le x}\omega^*(n)^2=O\left(x(\log x)^2\right),$$
where $B$ is a constant. Motived by Prachar's work, Erd\H os and Prachar \cite{EP} proved that the number of pairs of primes $p$ and $q$ so that the least common multiple $[p-1,q-1]\le x$ is bounded by $O(x\log\log x)$. Following a remark of Erd\H os and Prachar, R. M. Murty and V. K. Murty \cite{MM} improved this to $O(x)$. By this improvement, they reached the nice bounds
$$x(\log\log x)^3\ll\sum_{n\le x}\omega^*(n)^2\ll x\log x.$$
With these in hands, R. M. Murty and V. K. Murty conjectured that there is some positive constant $C$ such that
$$\sum_{n\le x}\omega^*(n)^2\sim Cx\log x$$
as $x\rightarrow \infty$.
In this note, the author shall give a slight improvement of the result due to R. M. Murty and V. K. Murty towards the correct direction of their conjecture.
\begin{theorem}\label{thm1} There are two absolute constants $a_1$ and $a_2$ such that $$a_1x\log x\le \sum_{n\le x}\omega^*(n)^2\le a_2x\log x.$$
\end{theorem}
\begin{proof}
We only need to prove the lower bound as the upper bound is displayed by R. M. Murty and V. K. Murty.
Throughout our proof, the number $x$ is sufficiently large.
From the paper of R. M. Murty and V. K. Murty \cite[equation (4.10)]{MM}, we have
\begin{align}\label{eq1}
\sum_{n\le x}\omega^*(n)^2=x\sum_{d\le x}\varphi(d)\bigg(\sum_{\substack{p\le x\\p\equiv 1\!\!\!\pmod{d}}}\frac{1}{p-1}\bigg)^2+O(x).
\end{align}
Integrating by parts gives
\begin{align}\label{eq2}
\sum_{\substack{p\le x\\p\equiv 1\!\!\!\pmod{d}}}\frac{1}{p}=\frac{\pi(x;d,1)}{x}+\int_{2}^{x}\frac{\pi(t;d,1)}{t^2}dt\ge \int_{x^{3/4}}^{x}\frac{\pi(t;d,1)}{t^2}dt,
\end{align}
where $\pi(t;d,1)$ is the number of primes $p\equiv 1\pmod{d}$ up to $t$.
Thus, from equations (\ref{eq1}) and (\ref{eq2}) we obtain
\begin{align}\label{eq3}
\sum_{n\le x}\omega^*(n)^2\ge x\sum_{d\le x^{1/3}}\varphi(d)\left(\int_{x^{3/4}}^{x}\frac{\pi(t;d,1)}{t^2}dt\right)^2+O(x).
\end{align}
For any integer $0\le j\le \left\lfloor \frac{\log x}{13\log 2}\right\rfloor$, let $Q_j=2^jx^{1/4}$. Then $Q_j<x^{1/3}$ for all integers $j$.
From a weak form of the Bombieri--Vinogradov theorem (see for example \cite{Da}), we have
\begin{align*}
\sum_{Q_j<d\le 2Q_j}\max_{y\le z}\left|\pi(y;d,1)-\frac{\text{li}~y}{\varphi(d)}\right|\ll \frac{z}{(\log z)^5},
\end{align*}
for any $0\le j\le \left\lfloor \frac{\log x}{13\log 2}\right\rfloor$ and $x^{3/4}\le z\le x$, where the implied constant is absolute. It follows immediately that
\begin{align}\label{eq4}
\max_{y\le z}\left|\pi(y;d,1)-\frac{\text{li}~y}{\varphi(d)}\right|<\frac{\text{li}~z}{\varphi(d)\log z}
\end{align}
hold for all $Q_j<d\le 2Q_j$ but at most $O\left(Q_j/(\log x)^2\right)$ exceptions. From equation (\ref{eq4}) we have
\begin{align*}
\pi(y;d,1)>\frac{\text{li}~y}{2\varphi(d)} \quad \left(z/2<y\le z\right)
\end{align*}
for all $Q_j<d\le 2Q_j$ with at most $O\left(Q_j/(\log x)^2\right)$ exceptions. A little thought with the dichotomy of $z$ between the interval $[x^{3/4},x]$ leads to the fact
\begin{align}\label{eq5}
\pi(y;d,1)>\frac{\text{li}~y}{2\varphi(d)}> \frac{y}{3\varphi(d)\log y} \quad \left(\forall~ x^{3/4}\le y\le x\right)
\end{align}
for all $Q_j<d\le 2Q_j$ except for $O\left(Q_j/\log x\right)$ exceptions. Let $S_j$ be the set of all integers $Q_j<d\le 2Q_j$ such that equation (\ref{eq5}) hold. Thus, from the analysis above and equation (\ref{eq3}), (\ref{eq5}) we conclude that
\begin{align*}
\sum_{n\le x}\omega^*(n)^2&\ge \frac{x}{9}\sum_{0\le j\le \left\lfloor \frac{\log x}{13\log 2}\right\rfloor}\sum_{\substack{d\in S_j}}\varphi(d)\left(\int_{x^{3/4}}^{x}\frac{1}{\varphi(d)t\log t}dt\right)^2+O(x)\nonumber\\
&\gg x\sum_{0\le j\le \left\lfloor \frac{\log x}{13\log 2}\right\rfloor}\sum_{\substack{d\in S_j}}\frac1{\varphi(d)}\gg x\log x.
\end{align*}
\end{proof}
It is worth here mentioning that we have the following corollary
$$\sum_{p,q\le x}\frac{1}{[p-1,q-1]}\asymp  \log x$$
due to (see \cite[page 6, last line]{MM})
$$\sum_{n\le x}\omega^*(n)^2=\sum_{p,q\le x}\frac{x}{[p-1,q-1]}+O(x).$$  
\section*{Acknowledgments}
The author is supported by the Natural Science Foundation of Jiangsu Province of China, Grant No. BK20210784, China Postdoctoral Science Foundation, Grant No. 2022M710121. He is also supported by foundation numbers JSSCBS20211023 and YZLYJF2020PHD051.

\end{document}